\documentclass[reqno, 12pt]{amsart}
\pdfoutput=1
\makeatletter
\let\origsection=\section \def\section{\@ifstar{\origsection*}{\mysection}}
\def\mysection{\@startsection{section}{1}\z@{.7\linespacing\@plus\linespacing}{.5\linespacing}{\normalfont\scshape\centering\S}}
\makeatother

\usepackage{amsmath,amssymb,amsthm}
\usepackage{mathrsfs}
\usepackage{mathabx}\changenotsign
\usepackage{dsfont}
\usepackage{hyperref}

\usepackage{xcolor}

\usepackage[open,openlevel=2,atend]{bookmark}

\usepackage[abbrev,msc-links,backrefs]{amsrefs}
\usepackage{doi}

\renewcommand{\PrintDOI}[1]{\doi{#1}}

\usepackage[T1]{fontenc}
\usepackage{lmodern}
\usepackage[babel]{microtype}
\usepackage[english]{babel}

\usepackage{geometry}
\numberwithin{equation}{section}
\numberwithin{figure}{section}

\usepackage{enumitem}

\let\polishlcross=\l
\def\l{\ifmmode\ell\else\polishlcross\fi}

\let\backslash=\smallsetminus

\makeatletter
\def\moverlay{\mathpalette\mov@rlay}
\def\mov@rlay#1#2{\leavevmode\vtop{   \baselineskip\z@skip \lineskiplimit-\maxdimen
   \ialign{\hfil$\m@th#1##$\hfil\cr#2\crcr}}}
\newcommand{\charfusion}[3][\mathord]{
    #1{\ifx#1\mathop\vphantom{#2}\fi
        \mathpalette\mov@rlay{#2\cr#3}
      }
    \ifx#1\mathop\expandafter\displaylimits\fi}
\makeatother

\DeclareFontFamily{U}  {MnSymbolC}{}
\DeclareSymbolFont{MnSyC}         {U}  {MnSymbolC}{m}{n}
\DeclareFontShape{U}{MnSymbolC}{m}{n}{
    <-6>  MnSymbolC5
   <6-7>  MnSymbolC6
   <7-8>  MnSymbolC7
   <8-9>  MnSymbolC8
   <9-10> MnSymbolC9
  <10-12> MnSymbolC10
  <12->   MnSymbolC12}{}
\DeclareMathSymbol{\powerset}{\mathord}{MnSyC}{180}

\let\epsilon=\varepsilon

\let\rho=\varrho
\let\theta=\vartheta

\let\phi=\varphi

\theoremstyle{plain}
\newtheorem{thm}{Theorem}[section]

\newtheorem{prop}[thm]{Proposition}

\newtheorem{cor}[thm]{Corollary}
\newtheorem{lem}[thm]{Lemma}

\theoremstyle{definition}

\newtheorem{question}[thm]{Question}

\usepackage{accents}

\usepackage{lineno}

\begin{document}

\title{Minimal Ramsey graphs for cyclicity}
\author{
Damian Reding
\and Anusch Taraz
}

\address{Technische Universit\"at Hamburg, Institut f\"ur Mathematik, Hamburg, Germany}
\email{\{damian.reding|taraz\}@tuhh.de}

\maketitle

\begin{abstract}

We study graphs with the property that every edge-colouring admits a monochromatic cycle (the length of which may depend freely on the colouring) and describe those graphs that are minimal with this property. We show that every member in this class reduces recursively to one of the base graphs $K_5-e$ or $K_4\vee K_4$ (two copies of $K_4$ identified at an edge), which implies that an arbitrary $n$-vertex graph with $e(G)\geq 2n-1$ must contain one of those as a minor. We also describe three explicit constructions governing the reverse process. As an application we are able to establish Ramsey infiniteness for each of the three possible chromatic subclasses $\chi=2, 3, 4$, the unboundedness of maximum degree within the class as well as Ramsey separability of the family of cycles of length $\leq l$ from any of its proper subfamilies. 
\end{abstract}


\linespread{1.3}

\section{Introduction and results}

By an $r$-\textit{Ramsey graph for} $H$ we mean a graph $G$ with the property that every $r$-edge-colouring of $G$ admits a monochromatic copy of $H$. Wo focus on the Ramsey graphs that are \textit{minimal} with respect to the subgraph relation, i.e. no proper subgraph is a Ramsey graph for $H$. As a consequence of Ramsey's theorem~\cite{RAM} such graphs always exist. Minimal Ramsey graphs, their constructions, number on a fixed vertex set, connectivity as well as extent of chromatic number and maximum degree have been investigated by Burr, Erd\H{o}s and Lov\'asz~\cite{BEL}, Ne\v{s}et\v{r}il and Rödl~\cite{NRC}, Burr, Faudree and Schelp~\cite{BFS} as well as Burr, Ne\v{s}et\v{r}il, Rödl~\cite{BNR} and others. More recently, the question of the minimum degree of minimal Ramsey graphs initiated by Burr, Erd\H{o}s, Lov\'asz~\cite{BEL} was picked up again by Fox, Lin~\cite{FL} and Szab\'o, Zumstein and Z\"urcher~\cite{SZZ}. Subsequently Fox, Grinshpun, Liebenau, Person and Szab\'o~\cite{FGLPS} have employed the parameter in a proof of Ramsey non-equivalence (or separability)~\cite{FGLPS} and also obtained some generalizations to multiple colours~\cite{FGLPSM}.

However, a persistent obstacle is that the structure of (minimal) Ramsey graphs for a specific graph $H$ is difficult to characterize, essentially because it requires a practical description of how graphs edge-decompose into $H$-free subgraphs. Indeed, few exact characterizations are known other than some simple ones for stars and collections of such~\cite{BEL}.

The obstacle turns out to be a lesser one if $H$ is relaxed to be a graph property. We say that a graph $G$ is an $r$-\textit{Ramsey graph for a graph property $\mathcal{P}$} (which is closed under taking supergraphs), if every $r$ edge-colouring of $G$ admits a monochromatic copy of a member of $\mathcal{P}$. The choice of the member is thus allowed to depend freely on the choice of colouring. We denote that class by $\mathcal{R}_r(\mathcal{P})$ and the subclass of minimal ones by $\mathcal{M}_r(\mathcal{P})\subset\mathcal{R}_r(\mathcal{P})$.

Indeed, this is not a far-fetched definition. Results on the corresponding notion of Ramsey numbers for graph properties appear across the literature both in and outside the context of Ramsey theory, e.g. connectivity~\cite{M}, minimum degree~\cite{KS}, planarity~\cite{B}, the contraction clique number~\cite{T} or, more recently, embeddability in the plane~\cite{FKS}. For a small number of such properties, the minimal order $R_r(\mathcal{P})$ of a Ramsey graph for $\mathcal{P}$ is known exactly, e.g. $R_r(\chi\geq k)=(k-1)^r+1$~\cite{LZ}. Most notable, however, is the characterization of the chromatic Ramsey number of $H$ as the Ramsey number for the graph property $Hom(H)$ by Burr, Erd\H{o}s and Lov\'asz~\cite{BEL}. The notion also connects naturally to classical graph parameters. Indeed, for every number $r\geq 2$ of colours we have that $G\in\mathcal{R}_r(\mathcal{C}_{\text{odd}})$, where $\mathcal{C}_{\text{odd}}$ denotes the property of containing an odd cycle, if and only if $\chi (G)\geq 2^r+1$ (for the \emph{if}-direction, note that if $G\notin\mathcal{R}_r(\mathcal{C}_{\text{odd}})$, then $G$ edge-decomposes into $\leq r$ bipartite graphs, whence a proper $2^r$-colouring of $V(G)$ is given by the $r$-tuples of $0$'s and $1$'s. The \emph{only if}-direction follows by a simple inductive argument on $r\geq 1$). Consequently we have that $G\in\mathcal{M}_r(\mathcal{C}_{\text{odd}})$ if and only if $G$ is minimal subject to $\chi(G)\geq 2^r+1$, so the study of $\mathcal{M}_r(\mathcal{C}_{\text{odd}})$ is precisely the study of the well-known notion of $(2^r+1)$-\emph{critical} graphs.

The property we focus on in this paper is the property $\mathcal{C}$ of containing an arbitrary cycle. Indeed we have the following useful characterization of $\mathcal{R}_r(\mathcal{C})$ (and hence of $ \mathcal{M}_r(\mathcal{C})$) in terms of local edge-densities of subgraphs.

\begin{prop}\label{prop:NW}
For every integer $r\geq 2$, we have that $G\in\mathcal{R}_r(\mathcal{C})$ if and only if $\frac{e(H)-1}{v(H)-1}\geq r$ for some subgraph $H\subseteq G$, and consequently we have that $G\in\mathcal{M}_r(\mathcal{C})$ if and only if both $\frac{e(G)-1}{v(G)-1}=r$ and $\frac{e(H)-1}{v(H)-1}<r$ for every proper subgraph $H\subset G$.
\end{prop}

Since the graphs in $\mathcal{R}_r(\mathcal{C})$ are precisely those which do not edge-decompose into $r$ forests, one obtains Proposition \ref{prop:NW} as a direct translation of the following well-known theorem.

\begin{thm}\label{thm:NW}
(\text{Nash-Williams' Arboricity Theorem}~\cite{NW})
Every graph $G$ admits an edge-decomposition into $\left\lceil ar(G)\right\rceil$ many forests, where $ar(G):=\max_{J\subseteq G, v_J>1}\frac{e_J}{v_J-1}$.
\end{thm}

We remark that this is not the first time that Theorem \ref{thm:NW} finds use in graph Ramsey theory, see e.g.~\cite{PR} for an account of how the theorem can be used to establish the relation $ar(G)\geq r\cdot ar(F)$ for every $r$-Ramsey graph $G$ of an arbitrary graph $F$.

For the rest of the paper we focus on the case $r=2$ and also write $\mathcal{R}(\mathcal{C}):=\mathcal{R}_2(\mathcal{C})$ and $\mathcal{M}(\mathcal{C}):=\mathcal{M}_2(\mathcal{C})$. Given the aforementioned relation between $\mathcal{M}(\mathcal{C})$ and $5$-critical graphs, the latter of which are completely described (in the language of \emph{constructibility}) by the well-known H\'ajos construction~\cite{HJS} originating in the single base graph $K_5$, one might suspect that a similar reduction to base graphs is possible for $\mathcal{M}(\mathcal{C})$. Indeed, our first result does just that. Our two base graphs will be $K_5-e\in\mathcal{M}(\mathcal{C})$
and $K_4\vee K_4\in\mathcal{M}(\mathcal{C})$, the graph obtained by identifying two copies of $K_4$ at an edge; a quick computation based on Proposition \ref{prop:NW} shows that these are in $\mathcal{M}(\mathcal{C})$.

\begin{thm}\label{thm:first}
For every $G\in \mathcal{M}(\mathcal{C})$ there exists $n\in\mathbb{N}_0$ and a sequence $G_k$ of minimal Ramsey graphs for $\mathcal{C}$ such that
$$\{K_5-e, K_4\vee K_4\}\ni G_0\prec G_1\prec\ldots\prec G_{n}=G,$$
where $\prec$ denotes the minor relation. In fact, for every $k\in [n]$ one can take $G_{k-1}$ to be an arbitrary minimal Ramsey subgraph (for $\mathcal{C}$) of the Ramsey graph (for $\mathcal{C}$) obtained from $G_k$ by contracting an arbitrary edge that belongs to at most one triangle in $G_{k}$.
\end{thm}

As we shall show, the contraction of an edge, which is in at most one triangle, preserves the Ramsey property of a Ramsey-graph for $\mathcal{C}$, whence a minimal Ramsey-subgraph can be found. The theorem guarantees that continuing the reduction in this way necessarily results in $K_5-e$ or $K_4\vee K_4$. By combining \ref{prop:NW} with \ref{thm:first} we therefore obtain:

\begin{cor}
Every graph $G$ with $e(G)\geq 2v(G)-1$ contains one of $K_5-e$, $K_4\vee K_4$ as a minor.
\end{cor}

Upon reinterpretation of Theorem \ref{thm:first}, every $G\in\mathcal{M}(\mathcal{C})$ can be obtained by starting with one of the two base graphs by recursively splitting a vertex of a suitable supergraph. A concrete description of the process would result in an algorithm constructing all minimal Ramsey-graphs for $\mathcal{C}$. Traditionally, for graphs $H$ such extensions were done by means of \emph{signal senders}, i.e. non-Ramsey graphs $G$ with two special edges $e$ and $f$, which attain same (respectively distinct) colours in every $H$-free colouring, which were then use to establish infiniteness of $\mathcal{M}(H)$ and much more, see e.g.~\cite{BEL} and~\cite{BNR}. However, it follows from an extension of Theorem \ref{thm:NW} by Reiher and Sauermann~\cite{RS} that no (positive) signal senders for $\mathcal{C}$ can exist: indeed, given a graph $G$ that edge-decomposes into two forests, for any choice of $e$ and $f$ one finds an edge-decomposition with $e$ and $f$ belonging to different colour classes. Instead, one may prove infiniteness for $\mathcal{M}(\mathcal{C})$ by noting (by an argument similar to that in~\cite{ADOV}) that a $4$-regular graph of girth $g$ (which is known to exist by~\cite{ESA}) must contain a minimal Ramsey graph for cyclicity, where the monochromatic cycles are of length $\geq g$.

Our second result provides a much simpler way to make progress towards this aim by describing three entirely constructive ways to enlarge a graph in $\mathcal{M}(\mathcal{C})$ that allow to track its structure; note that the first increases the number of vertices by $1$, while the other two increase it by $2$.

\begin{thm}\label{thm:second}
If $G\in\mathcal{M}(\mathcal{C})$, then also $G^*\in\mathcal{M}(\mathcal{C})$, where $G^*$ is a larger graph obtained from $G$ by applying one of the following three constructions:
\begin{enumerate}
\item Given a $2$-path $uvw$ in $G$, do the following: Introduce a new vertex $x$. Join $x$ to each of $u, v$ and $w$. Then delete edge $vw$.
\item Given an edge $vw$ in $G$, do the following: Introduce a new vertex $x$. Join $x$ to both $v$ and $w$. Then apply construction (1) to the $2$-path $xvw$.
\item Given a $2$-path $uvw$ in $G$, do the following: apply construction (1) to $uvw$ and $wvu$ at the same time, that is: Introduce new vertices $x, y$. Join both $x, y$ to each of $u, v, w$. Then delete edges $uv$ and $vw$.
\end{enumerate}
\end{thm}

Note that one has $\chi(G)\leq 4$ for every graph $G\in\mathcal{M}(\mathcal{C})$ or, more generally $\chi(G)\leq 2r$ for every graph $G\in\mathcal{M}_r(\mathcal{C})$. Indeed, any $n$-vertex graph $G\in\mathcal{M}_r(\mathcal{C})$ contains a subgraph $H$ with $\delta(H)\geq\chi(G)-1$, which at the same time satisfies $\delta(H)\leq d(H)\leq\frac{2[r(n-1)+1]}{n}<2r$, where $d(H)$ denotes the average degree of $H$. Our Theorem \ref{thm:second} now implies:

\begin{cor}\label{cor:inf}
Each of the three partition classes of $\mathcal{M}(\mathcal{C})$ corresponding to chromatic number $\chi = 2, 3, 4$, respectively, consists of infinitely many pairwise non-isomorphic graphs.
\end{cor}

In fact, since our first two constructions can be seen to preserve planarity, infinitely many of the above graphs with $\chi = 2, 3$ can be chosen planar each. On the other hand, the smallest bipartite graph $G\in\mathcal{M}(\mathcal{C})$ is already $K_{3, 5}$ (obtained as $K_5-e\longrightarrow (K_{2, 3})^+\longrightarrow (K_{2, 4})^+\longrightarrow K_{3, 5})$. Since $e(G)>2v(G)-4$, any such must be non-planar.\\

Note that the fact that $\chi(G)\leq 4$ for $G\in\mathcal{M}(\mathcal{C})$ is much unlike the situation for graphs $G\in\mathcal{M}(H)$ for $H=K_3$ or $H$ $3$-connected, where $\chi(G)$ becomes arbitrarily large (see~\cite{BNR}) and hence so does $\Delta(G)$. Despite the boundedness of $\chi(G)$ we are still able to show:

\begin{cor}\label{cor:delta}
For every $\Delta\geq 1$ there exists $G\in\mathcal{M}(\mathcal{C})$ with $\Delta (G)\geq\Delta$.
\end{cor}

Indeed, Corollary \ref{cor:delta} is a special case of a much more general theorem, which as an exhaustive application of \ref{thm:second} asserts that the structure of $\mathcal{M}(\mathcal{C})$ is actually quite rich.

By a \emph{forest of cycles} we refer to a graph $F$ obtained, with disregard to isolated vertices, by starting with a cycle and then recursively adjoining a further cycle by identifying at most one of its vertices with a vertex on already existing cycles. Clearly there are forests of cycles of arbitrarily large maximum degree. Note that thanks to every edge of $F$ belonging to precisely one cycle, we can $2$-edge-colour a forest of cycles $F$ in such a way that every cycle in $F$ is monochromatic while choosing each cycle's colour independently of that of any other cycle. Call any such colouring \emph{cycle-monochromatic}.

\begin{thm}\label{thm:third}
For every forest of cycles $F$ and every integer $n\geq 5$ satisfying $n\geq \left|F\right|$ there exists $G\in\mathcal{M}(\mathcal{C})$ with the following properties:
\begin{enumerate}
\item $\left|G\right|=n$
\item $F$ is a subgraph of $G$
\item Every cycle-monochromatic $2$-edge-colouring of $F$ extends to a $2$-edge-colouring of $G$, in which there are no monochromatic cycles other than those already in $F$.
\end{enumerate}
\end{thm}

Note that the condition $n\geq \left|F\right|$ could be replaced by $n=\left|F\right|$ if the definition of a forest of cycles were relaxed so as to allow isolated vertices, but this variant would somewhat undermine the strength of the statement.

Since, as is quickly seen, a forest of cycles $F$ on $n$ (non-isolated) vertices contains between $n$ and $\frac{3}{2}(n-1)$ edges, Theorem \ref{thm:third} also guarantees that any such $F$ ($n\geq 5$) extends to some $G\in\mathcal{M}(\mathcal{C})$ with $F$ as a spanning subgraph by adding only $k$ edges, where $\frac{1}{2}(v(F)+1)\leq k\leq v(F)-1$. Finally, we remark on a second corollary of \ref{thm:third}.


\begin{cor}\label{cor:equiv}
For all $l \geq 4$ the family $\{C_3,\ldots, C_l\}$ is not Ramsey-equivalent to any proper subfamily of itself, that is, for every proper $\mathcal{F}\subset\{C_3,\ldots, C_l\}$ there exists a (minimal) Ramsey-graph for $\{C_3,\ldots, C_l\}$, which is not a Ramsey-graph for $\mathcal{F}$.
\end{cor}

Corollary \ref{cor:equiv} asserts that for every $l\geq 3$ the cycle family $\mathcal{F}:=\{C_3,\ldots, C_l\}$ and any proper subfamily $\mathcal{F}_0$ of $\mathcal{F}$ are Ramsey-separable (or Ramsey non-equivalent). The concepts were introduced in~\cite{SZZ} and subsequently studied in e.g.~\cite{FGLPS},~\cite{ARU} and~\cite{BL}. A central open problem in the area is whether some two distinct graphs are Ramsey equivalent. The existence of Ramsey graphs for cycles $C_k$ with girth $k$ (which follows from the Random Ramsey Theorem, see also~\cite{HRRS}) sorts out this question in the case of single cycles and also cycle families $\mathcal{F}_0$ containing the longest cycle $C_l$ of $\mathcal{F}$. In contrast, \ref{cor:equiv} provides constructively a supply of separating Ramsey graphs for all proper $\mathcal{F}_0$.\\





The organization of the paper is as follows. In each of the following three sections we provide the proofs of Theorem \ref{thm:first}, Theorem \ref{thm:second} and Theorem \ref{thm:third}, respectively, and subsequently discuss the possibility of some generalizations in the concluding remarks.

\section{Proof of theorem \ref{thm:first}}

Our proof of \ref{thm:first} relies on three lemmas. We state the elementary one first, which holds for any number of colours.

\begin{lem}\label{lem:elem}
Every $G\in\mathcal{M}_r(\mathcal{C})$ satisfies $r+1\leq\delta (G)\leq 2r-1$ and is also $2$-connected. 
\end{lem}

\begin{proof}
An immediate consequence of Proposition \ref{prop:NW} is that every $G\in\mathcal{M}_r(\mathcal{C})$ has size $e(G)=rv(G)-(r-1)$ and every subgraph $H\subseteq G$ has average degree $d(H)<2r$, which implies the upper bound for $\delta (H)$ (including the case $H=G$). For the lower bound for $\delta (G)$ suppose that $G$ contains a vertex $v$ of degree at most $r$. Colour the outgoing edges with distinct colours; since now no monochromatic cycle can pass through $v$, it follows that $G-v$ itself must be Ramsey for $\mathcal{C}$, thus contradicting the minimality of $G$. For connectivity suppose that $G$ can be disconnected by removing at most one vertex, so $G$ consists of two proper subgraphs $G_1, G_2$ which may or may not have a vertex in common. Since removing an edge from $G_1$ destroys the Ramsey property of the whole graph, we can fix an $r$-edge-colouring of $G_2$ without a monochromatic cycle. It follows that $G_1$ itself must be Ramsey for $\mathcal{C}$, again contradicting the minimality of $G$.
\end{proof}

In the following we assume that $r=2$. The following lemma asserts that contraction of certain edges preserves the Ramsey property for cyclicity.

\begin{lem}\label{contract}
If $G\in\mathcal{R}(\mathcal{C})$, then $G/e\in\mathcal{R}(\mathcal{C})$, where $G/e$ is the graph obtained from $G$ by contracting an arbitrary edge $e\in E(G)$ that lies in at most one triangle.
\end{lem}

\begin{proof}
Let $e$ be as above and fix a $2$-edge-colouring of $G/e$.\\

\emph{Case 1.} If $e$ belongs to no triangle in $G$, then a $2$-edge colouring of $G/e$ induces a $2$-edge colouring of $G-e$, and any monochromatic cycle in $G-e$ induces a monochromatic cycle in $G/e$. If there is no monochromatic cycle in $G-e$, then, by Ramseyness of $G$, rejoining $e$ produces a monochromatic cycle irrespective of its colour. So $G-e$ must contain both a blue and red path joining the vertices of $e$. Note that since these are edge-disjoint, at least one of the paths must have length at least $3$, otherwise $e$ would be chord to a four-cycle. Hence there is a monochromatic cycle in $G/e$.\\

\emph{Case 2.} If $e$ belongs to one triangle in $G$, then a $2$-edge-colouring of $G/e$ induces a $2$-edge colouring of $G-e$ with the other two triangle edges in the same colour. If $G-e$ has no monochromatic cycle, proceed as above. Suppose $G-e$ has a monochromatic cycle. If it does not use both of the other edges of the triangle containing $e$, then it induces a monochromatic in $G/e$. If the cycle does use both, so $e$ is a chord to the cycle, then it must be of length at least $5$ since $e$ is not chord to a four-cycle. But then again there is a path of length at least $3$ joining the vertices of $e$. Hence there is a monochromatic cycle in $G/e$. This completes the proof.
\end{proof}

Consequently, for graphs with every edge in at most one triangle, e.g. such with girth $\geq 4$, the property of being Ramsey for cyclicity is stable under arbitrary edge-contractions. Note that we could have dealt with case $2$ computationally by invoking Proposition \ref{prop:NW} (thus even obtaining that for $e$ in one triangle the Ramsey-graph $G/e$ is minimal whenever $G$ is) but a constructive proof sheds more light on the subject matter.

\begin{lem}\label{attached}
Any $2$-connected graph $G$ with every edge contained in at least two triangles satisfies $e(G)\geq 2v(G)$, unless $v(G)\leq 6$.
\end{lem}

\begin{proof}
We start with two simple observations:\\

(1) Since every edge of $G$ is chord to a $4$-cycle, we must have $\delta (G)\geq 3$. Note that wlog. we can assume that equality holds, because if $\delta (G)\geq 4$, then $e(G)\geq 2v(G)$ follows by the Handshaking Lemma. Suppose therefore that there is $v\in G$ with $d(v)=3$.\\

(2) Observe further that every vertex $v\in G$ with $d(v)=3$ necessarily lies in a $K_4$ in $G$. This is because each of the three edges incident to $v$ must be a chord of a $C_4$, which due to $d(v)=3$ is necessarily spanned by the other two.\\

Now fix both a $v\in G$ with $d(v)=3$ and a $K_:=K_4\subset G$ with $v\in K$.\\

\emph{Remark.} At this stage it is clear that the two base graphs $K_5-e$ and $K_4\vee K_4$ are the only graphs $G$ with $v(G)<7$, $\delta (G)=3$ and every edge chord of a $4$-cycle: this is clear when $v(G)=5$, and also when $v(G)=6$, since then $K_4\subset G$ with precisely $5$ more edges to built a further $K_4$ housing the remaining two vertices. (Hence, the two graphs also prove the lemma false when $v(G)<7$.)\\

Suppose $K$ is \emph{strongly attached} in $G$, that is, that some vertex $z$, say, outside of $K$ in $G$ is adjacent to at least two vertices $u, w$ in $K$. We choose the reduction of $G$ so that $G'$ also satisfies the hypothesis of the lemma with $v(G')=v(G)-1$ and $e(G')\leq e(G)-2$: Obviously $v$ is not adjacent to $z$, so $v\neq u$ and $v\neq w$. Let $t$ denote the fourth vertex in $K$; it may or may not be adjacent to $z$. Obtain $G'$ from $G$ by deleting $v$ and its three incident edges, and also add the edge between $t$ and $z$, if it does not exist already, so as to ensure that every edge of $G'$ is in at least two triangles. Note that $G'$ remains $2$-connected since clearly none of its vertices is a cutvertex.\\

Else, if $K$ is \emph{weakly attached} in $G$, that is, if every vertex of $G$ outside $K$ is adjacent to at most one vertex in the $K$, consider the following.

If $K$ does not contract to a cutvertex, then $G':=G/K$ clearly satisfies the hypothesis of the lemma with $v(G')=v(G)-3$ and $e(G')=e(G)-6$.

If $K$ does contract to a cutvertex $v$ in $G/K$, let $V_1, \ldots, V_k$ denote the vertex classes of the $k\geq 2$ connected components of $G/K-v$. Note that since $K$ is weakly attached we have that $n_i:=\left|V_i\right|\geq 3$ and that each of the subgraphs $G_i:=G[V_i\cup V(K)]$ satisfies the hypothesis of the lemma with $n_i+4=\left|V_i\cup V(K)\right|<v(G)$, so by induction we obtain
\begin{eqnarray*}
e(G) & = & e(G_1)+\ldots + e(G_k)-(k-1)e(K)\geq  2(n_1+4)+\ldots 2(n_k+4)-6k+6\\
& = & 2(n_1+\ldots +n_k)+8k-6k+6 = 2(v(G)-4)+2k+6\geq 2v(G)
\end{eqnarray*}

Note that the result now easily follows by induction on $v(G)$, provided it holds true in the cases $v(G)=7, 8, 9$:\\

For the cases $v(G)=8, 9$, consider as before a $K:=K_4\subset G$. If $K$ can be chosen strongly attached, we successfully reduce to the cases $v(G)=7, 8$. If not, then contracting a weakly attached $K$ necessarily results in either $K_5-e$ or $K_4\vee K_4$, with the contraction having occurred at one of its high degree vertices (else a strongly attached $K_4$ in the reduced graph must have already been strongly attached in $G$). Since each of the low degree vertices in the reduced graph is contained in a $K_4$ as well, the same $K_4$'s must have existed in $G$ prior contraction of $K$ or $K$ could not have been weakly attached. Consequently, $K$ intersects one of those $K_4$ at a cutvertex, thus contradicting $2$-connectedness.\\

The case $v(G)=7$ is more involved as we cannot reduce it to a smaller graph as in the previous cases: Suppose there exists a $2$-connected graph $G$ on $7$ vertices, with every edge occurring as the chord to a $4$-cycle, which satisfies $e(G)<2v(G)=14$. We now force a contradiction in several steps:

Fix a $K:=K_4$ in $G$ and let $v, u_1, u_2$ denote the $3$ vertices of $G$, which are not vertices of $K$. Since $G$ is $2$-connected, at least $2$ vertices of $K$ are incident to edges not in $K$, hence have degree $\geq 4$ in $G$. If any of these vertices has degree $\geq 5$, then the degree sum of $G$ is $\geq 5\cdot 3 + 4+5=24$. If, however, all of these have degree $=4$, then there must be at least $3$ vertices of degree $=4$ (since we cannot have an odd number of odd degree vertices), in which case the degree sum of $G$ is $\geq 4\cdot 3+3\cdot 4=24$. In any case, $G$ has at least $12$ edges. Hence, as $e(G)\leq 13$, $G$ is obtained from $K\cup\{v, u_1, u_2\}$ by adding $6$ or $7$ edges.

Note that since the degrees of $v, u_1, u_2$ are all $\geq 3$, but only $\geq 7$ edges can join $v, u_1, u_2$ to the vertices of $K$, the induced subgraph $H$ of $G$ on vertices $v, u_1, u_2$ contains at least $2$ edges. Wlog. suppose the edges are $u_1 v$ and $vu_2$ and further let $w$ be a vertex of $K$ adjacent to $v$. Note that at this stage there are at most $4$ more edges to add.

We claim that $u_1, u_2, v, w$ must form the vertices of a further $K_4$ in $G$. In that case, $G$ is obtained by adding at most one edge to the graph obtained by identifying $K$ with a further copy of $K_4$ at vertex $w$. This is a contradiction because if we do not add the edge, $G$ will not be $2$-connected, but if we do add the edge, it will not be chord to a $4$-cycle because its end vertices will only have $w$ as a common neighbour.

If $d(v)=3$, we are done, because $v$ is then contained in a $K_4$ with the remaining vertices necessarily given by the neighbours $u_1, u_2, w$ of $v$. If $d(v)\geq 4$, note that we must have $d(u_1)=3$ and $d(u_2)=3$. This follows since $2$ of $u_1, u_2, v$ must have degree $3$, otherwise $e(G-K)\geq(3+4+4)-e(H)\geq(3+4+4)-3>7$, a contradiction.

Hence, both $u_1$ and $u_2$ must lie in a $K_4$ (containing $v$) in $G$. Note that they must lie in the same $K_4$, otherwise the $K_4$ of $u_1$ and $v$ would take up $\geq 3$ of our remaining edges, thus leaving $\leq 1$ to be incident to $u_2$, in which case $d(u_2)\leq 2$, a contradiction. Hence $u_1, u_2, v$ lie in a $K_4$ in $G$, in particular $u_1$ and $u_2$ are adjacent. This leaves $\leq 3$ edges to build up $G$.

Assume, towards the final contradiction, that $w$ is not the fourth vertex of that $K_4$. Then, as $d(u_1)=3$ and $d(u_2)=3$, $w$ cannot be adjacent to $u_1$ or $u_2$. Since, however, the edge $wv$ is chord to a $4$-cycle, there must be two further vertices in $K$ that are adjacent to $v$. But then there remains at most one further edge to be incident to one of $u_1$ or $u_2$, in which case either $d(u_1)=2$ or $d(u_2)=2$, a contradiction.
\end{proof}

We are now ready to prove Theorem \ref{thm:first}.

\begin{proof}
Given $G\in\mathcal{M}(\mathcal{C})$, apply Lemma \label{contract} to a suitable edge and take a minimal Ramsey-subgraph of the resulting Ramsey-graph. Repeat this process until you end up with a graph $G_0$ with the property that every edge of $G$ is in at least two triangles. Since $G_0\in\mathcal{M}(\mathcal{C})$, so $e(G_0)=2v(G_0)-1$, we must have $v(G_0)\leq 6$ by Lemma \ref{attached}. The only such possibilities allowing no further contractions are $K_5-e$ and $K_4\vee K_4$ (the other such graphs on $6$ vertices all reduce to $K_5$-e as remarked above).
\end{proof}

\section{Proof of theorem \ref{thm:second}}

We partition Theorem \ref{thm:second} into three lemmas, each governing the effect of the respective operations on a graph in $\mathcal{M}(\mathcal{C})$, then show how they jointly imply Corollary \ref{cor:inf}.

\begin{lem}\label{lem:path}
If $G\in\mathcal{M}(\mathcal{C})$, then $G^*\in\mathcal{M}(\mathcal{C})$, where $G^*$ is the graph obtained from $G$ by applying construction (1) to an arbitrary $2$-path in $G$.
\end{lem}

\begin{proof}
The construction increases the number of vertices by $1$ and the number of edges by $2$, so $G^+$ retains the correct global density in order to be in $\mathcal{M}(\mathcal{C})$. Now, let $H^+\subset G^+$ be a proper subgraph and suppose wlog. that it uses the new vertex, so it uses at most two new edges. Then there exists a proper subgraph $H\subset G$ with $e(H^+)\leq e(H)+2$ and $v(H^+)=v(H)+1$, so
$$\frac{e(H^+)-1}{v(H^+)-1}\leq\frac{(e(H)+2)-1}{(v(H)+1)-1}=\frac{(e(H)-1)+2}{v(H)}<\frac{2(v(H)-1)+2}{v(H)}=2.$$
\end{proof}

Note that Lemma \label{lem:path} alone provides a constructive proof for the existence of infinitely many non-isomorphic minimal Ramsey-graphs for cyclicity. Indeed, applying this to $K_5-e$ in one of two possible ways (up to isomorphism), results in two further minimal Ramsey-graphs on $6$ vertices, one of which is the edge-maximal planar graph with one edge removed.


\begin{lem}\label{lem:diam}
If $G\in\mathcal{M}(\mathcal{C})$, then $G^{*}\in\mathcal{M}(\mathcal{C})$, where $G^{*}$ is the graph obtained from $G$ by applying construction (2) to an arbitrary edge in $G$.
\end{lem}

\begin{proof}
While Lemma \label{lem:diam} could be proved similarly to Lemma \label{lem:path} via Proposition \ref{prop:NW}, it is possible to provide an exhaustive graph-chasing proof, which may be of independent interest as it works in more generality. Note that the effect of construction (2) is the replacement of an edge by the diamond graph with the non-adjacent vertices taking the place of the ends of the original edge. We prove the lemma with the diamond replaced by any graph $D$, which admits two non-adjacent \emph{contact vertices} $c, d$ with the property that in any $2$-edge-colouring of $D$ without a monochromatic cycle there is a monochromatic path joining $c$ and $d$ (note that a graph in $\mathcal{M}(\mathcal{C})$ with an edge $cd$ removed already has this property). In particular, we prove the following claim.\\


\textsl{Claim.} \emph{If $G\in\mathcal{R}(\mathcal{C})$, then $G^{*}\in\mathcal{R}(\mathcal{C})$, where graph $G^{*}$ is obtained from $G$ via parallel composition of $G-e$ with $D$ (that is, its contact edges taking the place of the ends of $e$). What's more, if $G\in\mathcal{M}(\mathcal{C})$ and $D$ is edge-minimal with the above property (given fixed contact vertices), then $G^{*}\in\mathcal{M}(\mathcal{C})$ as well.}\\

\textsl{Proof of Claim.} Fix a blue-red colouring of the edges of $G^{*}$. This restricts to a colouring of $G-e$; if this admits a monochromatic cycle, then so does $G^{++}$. Otherwise, since $G\in\mathcal{R}(\mathcal{C})$, there is both a red and a blue path in $G-e$ joining the contact vertices. One of these forms a monochromatic cycle in $G^{*}$ along with the monochromatic path in $D$, which must exist by definition whenever there is not already a monochromatic cycle in $D$.\\

Now suppose that both $G$ and $D$ are chosen minimal, in which case both clearly have minimal degree at least $2$. Given any edge $f$ of $G^{*}$ (so $f\neq e$), we show that in some colouring of $G^{*}-f$ there is no monochromatic cycle. If $f$ is an edge of $D$, such a colouring is obtained by fixing both a cycle-free colouring of $G-e$ and a cycle-free colouring of $D-f$ without a monochromatic path joining the contact vertices, and then inserting the coloured $D-f$ into the coloured $G-e$. If $f$ is an edge of $G$, fix both a cycle-free colouring of $G-f$ and a cycle-free colouring of $D$ with precisely one monochromatic path joining the contact vertices. If the path does not have the colour of $e$ in $G$, switch the colours in $D$. Now remove $e$ from the coloured $G-f$ and insert the coloured $D$. In the colouring of $G^{*}$ thus obtained there cannot be a monochromatic cycle. Suppose otherwise; then any monochromatic cycle would need to contain the whole monochromatic path in $D$ (as $G-f-e$ is coloured cycle-free) and since the contact vertices are non-adjacent, they would need to be joined by a path in $G-f-e$ of the colour of the path in $D$, and of length at least $2$. But along with $e$ any such path would form a monochromatic cycle in $G-f$. Contradiction.
\end{proof}

\begin{lem}
If $G\in\mathcal{M}(\mathcal{C})$, then $G^{*}\in\mathcal{M}(\mathcal{C})$, where $G^{*}$ is the graph obtained from $G$ by applying construction (3) to an arbitrary $2$-path in $G$.
\end{lem}

\begin{proof}
Let $G^*$ be the graph obtained from $G\in\mathcal{M}(\mathcal{C})$ by applying construction (3) to some path $uvw$. Since $e(G^*)=e(G)+4$ and $v(G^{*})=v(G)+2$, we have $G\in\mathcal{R}(\mathcal{C})$. To prove minimality, suppose that an edge $e$ is removed from $G^{*}$. Suppose that $e\notin E(G)$. In either case if $e$ is adjacent to $u$ or $w$ or if it is adjacent to $v$, proceed analogously as in the respective case in the proof of the previous lemma. Otherwise, if $e\in E(G)$, put a $2$-colouring on $E(G-e)$ and consider the colours of $uv$ and $vw$. Give the edges $ux, xv$ the colour of $uv$ and $uy$ the other colour. Also, give the edges $vy, yw$ the colour of $vw$ and $xw$ the other colour. If the $2$-colouring of $E(G-e)$ admits no monochromatic cycles, then neither does the so obtained $2$-colouring of $E(G^{*}-e)$.
\end{proof}

Finally, we are able to prove Corollary \ref{cor:inf}.

\begin{proof}
In order to obtain infinitely many graphs $G\in\mathcal{M}(\mathcal{C})$ with $\chi (G)=4$ fix a copy of $K_4$ in $K_5-e$ and let $e$ be an edge not belonging to that copy; now simply replace $e$ by a diamond, then replace an edge of that diamond by a diamond and so on. In order to obtain infinitely many $G\in\mathcal{M}(\mathcal{C})$ with $\chi (G)=3$ note that replacing every edge of any graph in $\mathcal{M}(\mathcal{C})$ results in precisely those graphs required. Finally, in order to obtain infinitely many graphs $G\in\mathcal{M}(\mathcal{C})$ with $\chi (G)=2$ start with $G_0:=K_{3, 5}\in\mathcal{M}(\mathcal{C})$ and repeatedly apply the following extension: apply construction (3) to some path $uvw$ in $G_i$ and let $x, y$ denote the two new vertices. Now apply construction (3) to the path $xvy$, thus producing two further vertices $x', y'$. Note that the resulting graph $G_{i+1}\in\mathcal{M}(\mathcal{C})$ is bipartite: Given a $2$-colouring on $V(G_i)$, give $x, y$ the colour of $v$ and $x', y'$ the other colour. (Alternatively note that any odd cycle, which may arise in the intermediate graph, must be using one of the edges $xv, yv$ and is thus destroyed in the construction of $G_{i+1}$.)
\end{proof}

\section{Proof of theorem \ref{thm:third}}

\begin{proof}
The proof is by induction on $n\geq 5$ and makes heavy use of constructions (1) and (2) as in \ref{thm:second}. For $n=5$ the result needs to be verified manually, and indeed $G=K_5-e$ works for all forests of cycles $F$ with $3\leq v(F)\leq 5$.

Let $x, y$ denote the non-adjacent vertices of $K_5-e$ and let $a, b, c$ denote the other three.

\begin{enumerate}
\item If w.l.o.g. $F$ is the red-coloured triangle $abc$, colour the edges $ay$ and $cx$ red and the remaining path $a-x-b-y-c$ blue.
\item If w.l.o.g. $F$ is the red-coloured $4$-cycle $a-b-c-x$, colour edge $cy$ red and the remaining path $x-b-y-a-c$ blue.
\item If w.l.o.g. $F$ is a red-coloured $C_5$, colour the remaining $4$-path blue.
\item If $F$ is a bowtie and the two triangles are of the same colour, colour the remaining $3$-path with the opposite colour.
\item If $F$ is a bowtie and the two triangles are of distinct colours, colour the remaining edges using each colour at least once.
\end{enumerate}

The aim in the induction step is to carefully build graphs in $\mathcal{M}(\mathcal{C})$ containing some prescribed forests of cycles from those containing some suitable smaller forest of cycles as provided by the induction hypothesis, while maintaining the possibility to extend the edge-colouring without creating new monochromatic cycles.\\

\textsl{Step 1 (Creating new space).} To begin with, we reduce the proof from $n\geq \left|F\right|$ to $n=\left|F\right|$. Fix $F$ and suppose $G\in\mathcal{M}(\mathcal{C})$ with $v(G)=v(F)$ is as in the statement of the theorem. We want to increase $G$ by one vertex while maintaining the containment of $F$ and the colouring extension property: Pick a vertex $v\in G$ with $d(v)=3$. Since $v(G)=v(F)$, such lies on precisely one cycle $C$ in $F$. Hence it is incident to an edge $vw$, which is not part of $C$ (even though $w$ may be); if $v$ is not in $F$, pick $vw\notin E(F)$, too. Further pick $u\in C$ such that $uv$ is an edge of $C$. Apply (1) to the path $u-v-w$, thus deleting the edge $vw$ and creating a new vertex $x$ incident to all of $u, v, w$. Note that by removing the edge $vw$ we have not destroyed any cycle of $F$ since thanks to $d(v)=3$, $vw$ is not an edge of $F$. Now given any $2$-edge-colouring of $G-F$ (or $G-F-vw$, respectively) as in the statement of the theorem, extend it by giving $xu$ and $xw$ arbitrary opposite colours and give $xv$ the colour opposite to that of $C$. If we have thus created a new monochromatic cycle, it has to pass through $x$, and hence, by choice of colouring, through $v$. This, however, is impossible since $v$ has maintained $d(v)=3$ throughout the construction. For the rest of the proof we can assume that $F$ is a spanning subgraph of the minimal Ramsey graph that contains it.\\

\textsl{Step 2 (Growing new trees).} We show how to extend the result for $F$ to that for $F$ with a disjoint triangle. Let $G\in\mathcal{M}(\mathcal{C})$ with $F\subset G$ and as in the statement of the theorem, and now without loss of generality $v(F)=v(G)$. Create new space in $G$ as in step 1, thus obtaining $G'$ with $v(G')=v(G)+1$ and the colouring property with respect to $F$ and fix the special edge-colouring of $G'-F$. Consider, as in step 1, the edge $xv$: Replace it by a diamond graph $D$ as in extension (2). Give the remaining so far uncoloured triangle in $D$, which is disjoint from $F$, a monochromatic colouring (this triangle is the new tree). If this is the colour of $xv$, give the two edges in $D$ now incident to $v$ distinct colours. If this is not the colour of $xv$, give the two edges in $D$ now incident to $v$ the colour of $xv$.\\

What we have so far achieved is that it suffices to prove the result for spanning trees of cycles. Note that any such can be obtained recursively by (1) starting with a triangle (2) enlarging it to required size (while it is a 'leaf' of the tree of cycles) (3) creating a required number of branches (that is, pairwise disjoint triangles) and repeating the procedure for any of the new branch triangles in turn. To complete the proof it therefore merely suffices to show how to enlarge cycles in $F$ irrespective of their distribution of attached branches, how to create a new triangle at a given vertex of degree $2$ in $F$ (\emph{extending an existing branch}), and finally, how to create a new triangle at a vertex, which is already used by more than one triangle (\emph{creating a new branch}).\\

\emph{Step 3 (Enlarging existing cycles).} Let $C$ be a cycle in $F$ to be enlarged and let $G\in\mathcal{M}(\mathcal{C})$ be for $F$ as in the statement of the theorem. Let $u-v-w$ be any $2$-path in $C$. Apply extension (1) as in Theorem \ref{thm:second}, thus producing a new vertex $x$ adjacent to all of $u, v, w$. The cycle $C$ is now enlarged in the resulting graph $G^+$ since $vw$ has been replaced by the $2$-path $v-x-w$. Any cycle-monochromatic $2$-edge-colouring $c$ of the enlarged forest $F^+$ now induces a cycle-monochromatic $2$-edge-colouring of $F$; pick a respective $2$-edge-colouring of $G-F$ and extend it to a respective colouring of $G^+-F^+$ by giving edge $xu$ the colour opposite of that of $xv$ in $c$.\\

\textsl{Step 4 (Extending existing branches).} Let $F\subset G$ be as before, and suppose that at $v\in F$ with $d(v)=2$ in $F$ a new triangle branch is to be created. Let $vw$ denote an edge not in $F$. Replace it by a diamond $D$, as before, and give the two edges in $D$ incident to $w$ distinct colours. Verifying the colouring property is now analogous to Step 2.\\

\textsl{Step 5 (Creating new branches).} Suppose that $u$ is a vertex of $F\subset G$, which lies in at least two triangles in $F$, and that a further triangle containing $u$ is to be created. Fix one of the triangles, which without loss of generality is a leaf to the tree of cycles, and label its remaining vertices $v$ and $w$. Apply (1) to $u-v-w$, thus destroying(!) one of the already existing triangles by removing edge $vw$, but instead creating the two new triangles $uvx$ and $uwx$, sharing edge $xu$. Apply now (1) again to the path $u-v-x$, thus destroying triangle $uvx$ by removing edge $vx$, but creating the new triangle $uvx'$, which is edge-disjoint from triangle $uwx$, and the extra edge $xx'$. Any cycle-monochromatic $2$-edge-colouring $c$ of the enlarged forest $F^+$ now induces a cycle-monochromatic $2$-edge-colouring of $F$; pick a corresponding special $2$-edge-colouring of $G-F$ and extend it to a special colouring of $G^+-F^+$ by giving edge $xx'$ the colour opposite to that of the triangles $uwx$ and $uvx'$ if these are monochromatic in $c$, and an arbitrary otherwise. This completes the proof.
\end{proof}

\section{Concluding Remarks}

In \ref{thm:first} we proved that every $G\in\mathcal{M}(\mathcal{C})$ can be obtained by starting with one of two base graphs by recursively splitting a vertex of a suitable supergraph. Any such description would shed light on how to constructively increase the girth while maintaining Ramseyness. This may be regarded as a first step towards the construction of Ramsey graphs for fixed length cycles $C_k$ with girth precisely $k$ (see e.g.~\cite{HRRS}, but to the best of our knowledge no explicit construction is known). We therefore raise the weaker question:

\begin{question}
For any $g\geq 3$, does there exist $G\in\mathcal{M}(\mathcal{C})$ with girth $g$?
\end{question}

We also note also how Lemma \ref{attached} implies that no minimal Ramsey-graph for $K_3$ is a minimial Ramsey-graph for $\mathcal{C}$ (since in the former every edge is in at least two triangles). It would be therefore interesting to work out what additional conditions on $G\in\mathcal{R}(\mathcal{C})$ ensure that $G\in\mathcal{R}(K_3)$. This might be possibly achieved by approximating the class $\mathcal{R}(K_3)$ by the classes $\mathcal{R}(\mathcal{C}_{\leq l})$ for fixed $l\geq 3$. Constructing graphs which are minimal with this property is probably hard as removing an edge and taking a good colouring gives rise to highly chromatic high-girth girth graphs (for which a non-recursive hypergraph-free construction was given only recently~\cite{A}). Note that similarly our remark in the introduction allows for a simple construction for $G\in\mathcal{R}_r(\mathcal{C}_{\text{odd}\leq l})$, just take $\chi (G)\geq 2^r+1$ and $g(G)\geq l$.\\




Another line of study relates to the fact that a $2$-edge-colouring of a Ramsey-graph for $K_3$ admits multiple monochromatic copies of $K_3$. As a step in this direction it therefore seems plausible to consider graphs with the approximative property that every $2$-edge-colouring admits either two disjoint monochromatic copies of $K_3$ in the same colour or a monochromatic cycle of length $\geq 4$. It is easy to see by case distinction that $G^{+}$, the graph obtained from some $G\in\mathcal{R}(C)$ by joining a new vertex to every vertex of $G$, has this property.\\

With regard to the existence of multiple monochromatic cycles, we observe that thanks to a known decomposition result into pseudoforests, see e.g.~\cite{PR}, one could in principle work out a theorem similar to ours for graphs, for which every $2$-edge-colouring admits a monochromatic connected graph containing at least two cycles. More generally, for $k\geq 1$ set $\mathcal{C}_k:=\{G:\; G\;\text{is connected and contains at least}\; k\;\text{cycles}\}$ and $m_k(G):=\frac{e(G)-1}{v(G)+k-2}$, excluding the trivial graphs. It is then easy to see that if $G$ contains a subgraph $H$ with $m_{k}(H)\geq r$, then $G$ is $r$-Ramsey for $\mathcal{C}_k$, and that if $G$ is minimal $r$-Ramsey for $\mathcal{C}_k$, then $m_k(H)<r$ for every proper subgraph $H\subset G$.


Crucial, however, to the characterization of graphs in $\mathcal{M}(\mathcal{C}_k)$ is the validity of the converse, which we do know about for $k\geq 3$. Indeed, with three available cycles allowing for circular arrangements, thus create new cycles, more complicated configuration may be needed in order for the Ramsey-property to be broken by the removal of any single edge. Instead, it seems more conceivable that the $+k$ in the density parameter is replaced by a larger quantity $f(k)$. To make this precise, for every $k\in\mathbb{N}$ let $f(k)$ denote the smallest natural number, if one exists, with the property that, for every integer $r\geq 1$, any graph $G$ satisfying $e(G)\leq r(v(G)+f(k)-2)$ edge-decomposes into at most $r$ subgraphs containing strictly less than $k$ (not necessarily edge-disjoint) cycles each. Note that $f$ is required to depend on $k$ only.

If $f(k)$ exist, then its are given by (the ceiling integer part of) the maximum of $\frac{e(G)}{r_k(G)}-v(G)+2$ taken over all graphs, where $r_k(G)$ denotes the size of a smallest edge-decomposition of $G$ into subgraphs with at most $k-1$ cycles. By the above, we know that $f(1)=1$ and $f(2)=2$. For $k\geq 3$ note that $f(k)\geq k$ holds by considering the chain of $k-1$ copies of triangles with two consecutive ones each identified at a vertex. We observe that for every $k$ the following are then equivalent:

\begin{enumerate}
\item $f(k):=\max\left\{\frac{e(G)}{r_k(G)}-v(G)+2:\; v(G)\geq 1\right\}<\infty$
\item $\forall r\in\mathbb{N}\backslash\{1\}$: $\mathcal{R}_r(\mathcal{C}_k)=\{G:\;\exists H\subseteq G:\; m_{f(k)}(H)\geq r\}$
\item $\forall r\in\mathbb{N}\backslash\{1\}$: $\mathcal{M}_r(\mathcal{C}_k)=\{G:\; m_{f(k)}(G)=r,\,\forall H\subset G, H\neq G:\; m_{f(k)}(H)< r\}$
\end{enumerate}

\begin{question}
For any $k\geq 3$, does $f(k)$ exists, that is, is $f(k)<\infty$? If so, what is $f(k)$?
\end{question}

Finally, we remark that cyclicity and $2$-connectivity are Ramsey equivalent and also that odd cyclicity and $3$-chromaticity are Ramsey equivalent. Undoubtedly, our results could therefore be generalized to both higher connectivity and chromaticity as well as to multiple colours.

We thank Dennis Clemens and Matthias Schacht for helpful comments.

\bibliography{cyclyfinal}

\end{document}